\documentclass{amsart} 
\usepackage{graphicx}
\usepackage{mathabx}
\usepackage{comment}
\usepackage{hyperref}
\usepackage[square, numbers, compress]{natbib}

\usepackage[dvipsnames,x11names,svgnames]{xcolor} 
\definecolor{celadon}{rgb}{0.67, 0.88, 0.69}
\colorlet{kollane}{LightGoldenrod1}
\colorlet{roheline}{celadon}
\colorlet{sinine}{LightSkyBlue2}
\usepackage[framemethod=tikz]{mdframed}

\newcommand{\N}{{\mathbb N}}

\newcommand{\ba}{{ba(\widetilde{M})}}

\DeclareMathOperator{\diam}{diam}

\DeclareMathOperator{\Lip}{Lip}

\newtheorem{thm}{Theorem}[section]

\newtheorem{lemma}[thm]{Lemma}
\newtheorem{prop}[thm]{Proposition}
\theoremstyle{definition}
\newtheorem{defn}[thm]{Definition}
\newtheorem{ex}[thm]{Example}
\newtheorem{qu}[thm]{Question}
\theoremstyle{remark}

\numberwithin{equation}{section}

\begin{document}

\title[Separating diameter $2$ properties from their weak$^*$ counterparts]{Separating diameter two properties\\ from their weak-star counterparts\\ in spaces of Lipschitz functions}%

\author{Rainis Haller}
\address{{\rm(R. Haller)} Institute of Mathematics and Statistics, University of Tartu, Narva mnt 18, 51009, Tartu, Estonia}%
\email{rainis.haller@ut.ee}

\author{Jaan Kristjan Kaasik}
\address{{\rm(J. K. Kaasik)} Institute of Mathematics and Statistics, University of Tartu, Narva mnt 18, 51009, Tartu, Estonia}%
\email{jaan.kristjan.kaasik@ut.ee}

\author{Andre Ostrak}
\address{{\rm(A. Ostrak)} Department of Mathematics, University of  Agder, Postbox 422 4604 Kristiansand, Norway.}%
\email{andre.ostrak@uia.no}%

\thanks{}%


\begin{abstract}
We address some open problems concerning Banach spaces of real-valued Lipschitz functions. Specifically, we prove that the diameter two properties differ from their weak-star counterparts in these spaces. In particular, we establish the existence of dual Banach spaces lacking the symmetric strong diameter two property but possessing its weak-star counterpart. We show that there exists an octahedral Lipschitz-free space whose bidual is not octahedral. Furthermore, we prove that the Banach space of real-valued Lipschitz functions from any infinite subset of $\ell_1$ possesses the symmetric strong diameter two property. These results are achieved by introducing new sufficient conditions, providing new examples and clarifying the status of known ones.

\end{abstract}
\maketitle

\section{Introduction}
We start by fixing some notation. Let $X$ be a nontrivial real Banach space. 
 We denote the closed unit ball, the unit sphere, and the dual space of $X$ by $B_X$, $S_X$, and $X^*$, respectively. 
A \emph{slice} of $B_{X}$ is a set of the form
\[
S(x^*, \alpha)\coloneq \{x\in B_{X}\colon x^*(x)>1-\alpha\},
\]
where $x^*\in S_{X^*}$ and $\alpha>0$. 
If $X$ is a dual space, then slices whose defining functional comes from the predual of $X$ are called weak$^*$ slices.

Let $M$ be a pointed metric space, i.e. a metric space with a fixed point $0$, with a distance function $d$. We denote by $\Lip_0(M)$ the Banach space of all Lipschitz functions $f\colon M\rightarrow\mathbb{R}$ satisfying $f(0)=0$ equipped with the norm
\[
    \|f\|=\sup\Big\{\frac{|f(x)-f(y)|}{d(x,y)}\colon x, y\in M,\ x\neq y\Big\}.
\]
Let $\delta\colon M\rightarrow \Lip_0(M)^*$ be the canonical isometric embedding of $M$ into $\Lip_0(M)^*$, which is given by $x\mapsto \delta_x$, where $\delta_x(f)=f(x)$. 
The norm closed linear span of $\delta(M)$ in $\Lip_0(M)^*$ is called the \emph{Lipschitz-free space} over $M$ and is denoted by $\mathcal{F}(M)$. It is well known that
$\mathcal{F}(M)^*= \Lip_0(M)$. For $x,y\in M$ with $x\neq y$, 
we define a \emph{molecule} $m_{x,y}$ as a norm one element $\tfrac{\delta_x-\delta_y}{d(x,y)}$ in $\mathcal{F}(M)$.

Set
\begin{equation*}\label{eq: Gamma:=(M times M) setminus ((x,x): x in M)}
    \widetilde{M} = \big\{ (x,y) \colon x,y\in M,\ x\neq y\big\}.
\end{equation*}
For a subset $A$ of $M$, we denote
\begin{equation*}\label{eq: GammaA}
    \Gamma_A = \big\{(x,y)\in \widetilde{M}\colon x\in A\text{ or } y\in A\big\}.
\end{equation*}

The de Leeuw's transform (see \cite[Definition 2.31 and Theorem 2.35]{MR3792558}) is a linear isometry defined by
\begin{align*}
    \Lip_0(M) \ni f\longmapsto \tilde{f} \in \ell_\infty(\widetilde{M}), \quad \text{where }\tilde{f}(x,y)= \frac{f(x)-f(y)}{d(x,y)}.
\end{align*}
As a side note, one may also regard the de Leeuw's transform as embedding $\Lip_0(M)$ into the Banach space of bounded continuous functions on $\widetilde M$ or into the Banach space of continuous functions on the Stone--\v Cech compactification of $\widetilde M$ (see \cite[p.~102]{MR3792558}).

Given a set $E$, we denote by $ba(E)$ the Banach space of bounded and finitely additive signed measures on the power set of $E$, where the norm of a measure $\mu\in ba(E)$ is its total variation $|\mu|(E)$.
Recall that $\ell_\infty(\widetilde M)^* = \ba$. Thus, for every $F\in \Lip_0(M)^*$, there exists $\mu\in \ba$ so that $\|F\|=|\mu|(\widetilde{M})$ and
\begin{equation*}
    F(f)=\int_{\widetilde{M}} \tilde{f} d\mu \qquad \text{for all $f\in \Lip_0(M)$.}
\end{equation*}

We study the following properties (\cite{MR3098474, MR4023351}) in spaces of Lipschitz functions. A Banach space $X$ has the
\begin{enumerate}
\label{def: d2p-s}
\item \emph{slice diameter $2$ property} (briefly, \emph{slice-D$2$P}) if every slice of $B_X$ has dia\-meter $2$;
\item \emph{diameter $2$ property} (briefly, \emph{D$2$P}) if every nonempty relatively weakly open subset
of $B_X$ has diameter $2$;
\item \emph{strong diameter $2$ property} (briefly, \emph{SD$2$P}) 
if every convex combination of slices of $B_X$ has diameter $2$, 
i.e. the diameter of $\sum_{i=1}^n \lambda_i S_i$ is $2$ 
whenever $n\in\mathbb N$, $\lambda_1,\dotsc,\lambda_n\geq 0$ with $\sum_{i=1}^n\lambda_i=1$, 
and $S_1,\dotsc,S_n$ are slices of $B_X$;
\item\emph{symmetric strong diameter $2$ property} (briefly, \emph{SSD$2$P}) if 
for every $n\in\mathbb N$, every family $\{S_1,\dotsc, S_n\}$ of slices of $B_X$, and every $\varepsilon>0$, 
there exist $f_1\in S_1,\dotsc, f_n\in S_n$, and $g\in B_{X}$ with $\|g\|>1-\varepsilon$ 
such that $f_i\pm g\in S_i$ for every $i\in \{1,\ldots,n\}$.
\end{enumerate}
If $X$ is a dual space, then we also consider the weak-star counterparts of these diameter two properties ($w^
*$-slice D$2$P, $w^*$-D$2$P, $w^*$-SD$2$P, and $w^*$-SSD$2$P), where slices and weakly open subsets in the above definitions are replaced by weak$^*$ slices and weak$^*$ open subsets, respectively. 
 It is known that $(4)\Rightarrow (3) \Rightarrow (2) \Rightarrow (1)$ and, with the exception of SSD$2$P, that these properties are distinct from their weak-star counterparts in dual Banach spaces. It has remained open whether SSD$2$P and $w^*$-SSD$2$P are equivalent for dual spaces (see {\cite[Question 6.2]{MR3917942}}).

Diameter two properties of $\Lip_0(M)$ have been studied
in \cite{MR3985517, MR3917942, HOP2022, zbMATH05159098,  MR4093788, MR4026495, MR4233633, MR3803112}. We recall some useful results but refer to the introduction of \cite{HOP2022} for a more detailed background on the topic. 

By \cite[Theorem 3.1]{MR3803112}, the space $\Lip_0(M)$ has the $w^*$-SD$2$P if and only if the metric space $M$ has the LTP.
\begin{defn}[{\cite[Theorem 3.1]{MR3803112}}]
    The metric space $M$ has the \emph{LTP} if, given a finite subset $N$ of $M$ and $\varepsilon>0$, there exist $u,v\in M$ with $u\neq v$ satisfying 
    \begin{equation}\label{ineq: LTP}
        (1-\varepsilon)\big(d(x,y)+d(u,v)\big)\leq d(x,u)+d(y,v)
    \end{equation}
    for all $x,y\in N$.
\end{defn}

By \cite[Theorem 2.1]{MR4026495}, the space $\Lip_0(M)$ has the $w^*$-SSD$2$P if and only if the metric space $M$ has the SLTP.
\begin{defn}[{\cite[Definition 1.3]{MR4026495}}]
    The metric space $M$ has the \emph{SLTP} if, given a finite subset $N$ of $M$ and $\varepsilon>0$, there exist $u,v\in M$ with $u\neq v$ satisfying \eqref{ineq: LTP} and
\begin{equation}\label{ineq: SLTP}
    \begin{aligned}
         &(1-\varepsilon)\big(2d(u,v)+d(x,y)+d(z,w)\big)\\
        &\qquad\qquad \qquad \leq d(x,u)+d(y,u)+d(z,v)+d(w,v)
     \end{aligned}
\end{equation}
for all $x,y,z,w\in N$.
\end{defn}

In \cite{HOP2022}, Haller, Ostrak, and Põldvere introduced stronger versions of the LTP and the SLTP called the seq-LTP and the seq-SLTP, respectively. 
\begin{defn}[{\cite[Definition 3.2]{HOP2022}}]\label{seq-LTP}
    The metric space $M$ has the \emph{seq-LTP} if there exist pairwise disjoint subsets $A_1,A_2,\dotsc$ of $M$ and elements $u_m,v_m\in A_m$ with $u_m\neq v_m$, $m\in \mathbb{N}$, satisfying 
\begin{equation}\label{ineq: seq-LTP}
    (1-\varepsilon)\big(d(x,y)+d(u_m,v_m)\big)\leq d(x,u_m)+d(y,v_m).
\end{equation}
for all $m\in \mathbb{N}$ and $x,y\in M\setminus A_m$.
\end{defn}

\begin{defn}[{\cite[Definition 2.2]{HOP2022}}]\label{seq-SLTP}
    The metric space $M$ has the \emph{seq-SLTP} if there exist pairwise disjoint subsets $A_1,A_2,\dotsc$ of $M$ and elements $u_m,v_m\in A_m$ with $u_m\neq v_m$, $m\in \mathbb{N}$, satisfying \eqref{ineq: seq-LTP} and
\begin{equation}\label{ineq: seq-SLTP}
    \begin{aligned}
         &(1-\varepsilon)\big(2d(u_m,v_m)+d(x,y)+d(z,w)\big)\\
        &\qquad\qquad \qquad \leq d(x,u_m)+d(y,u_m)+d(z,v_m)+d(w,v_m)
     \end{aligned}
\end{equation}
for all $m\in \mathbb{N}$ and $x,y,z,w\in M\setminus A_m$.
\end{defn}

\noindent By \cite[Theorem~3.3 (resp. Theorem~2.3)]{HOP2022}, if $M$ has the seq-LTP (resp. seq-SLTP), then $\Lip_0(M)$ has the SD$2$P (resp. SSD$2$P). It remained uncertain whether the seq-LTP (resp. seq-SLTP) of $M$ is necessary for $\Lip_0(M)$ to have the SD$2$P (resp. SSD$2$P).
Moreover, the authors of \cite{HOP2022} showed that there exists a metric space with the seq-LTP but not the SLTP, which proved that the SD$2$P and the SSD$2$P are different properties for the spaces of Lipschitz functions. Additionally, they constructed a metric space $M$ with the SLTP but without the seq-LTP (see Example \ref{Example2.7HOP} below), leaving open the question of whether $\Lip_0(M)$ has the SD$2$P or the SSD$2$P.
Furthermore, it remained open whether any of the diameter $2$ properties coincide with their weak-star counterparts in spaces of Lipschitz functions. 

We now outline the main results and the layout of the paper.

In section 2, we give new sufficient conditions for the space $\Lip_0(M)$ to have the SD$2$P and the SSD$2$P, which are strictly weaker than the seq-LTP and the seq-SLTP for $M$, respectively.

In section 3, we show that for the metric space $M$ constructed in \cite[Example 2.7]{HOP2022}, the space $\Lip_0(M)$ does not have the SSD$2$P. This observation establishes that the $w^*$-SSD$2$P differs from the SSD$2$P in spaces of Lipschitz functions, thereby answering {\cite[Question 6.2]{MR3917942}}.
In addition, we construct a metric space $M$ with the SLTP for which $\Lip_0(M)$ does not have the slice-D$2$P, thereby showing that the diameter $2$ properties are all different from their weak-star counterparts in spaces of Lipschitz functions. This also resolves \cite[Problem 1.1]{MR4093788} (see also \cite[p. 1681]{MR3985517}) as the Lipschitz-free space $\mathcal{F}(M)$ is octahedral but its bidual $\mathcal{F}(M)^{**}=\Lip_0(M)^*$ is not by \cite[Theorem 2.1 and Corollary 2.2]{MR3150166} (see also \cite[Theorem 2.3 and Theorem 2.4]{MR3346197}).

In section 4, we show that there exists an infinite metric subspace $M$ of $\ell_1$ without the seq-LTP. Nevertheless, we establish that the space $\Lip_0(M)$ has the SSD$2$P for all infinite metric subspaces $M$ of $\ell_1$. Previous studies have shown that all such metric subspaces $M$ have the LTP (see \cite[Proposition 4.7]{MR3803112}) and even the SLTP (see \cite[Example 3.3]{MR4026495}), indicating that $\Lip_0(M)$ has the $w^*$-SD$2$P and the $w^*$-SSD$2$P, respectively.

In section 5, we list some open problems.

\section{Sufficient conditions for strong diameter two properties in spaces of Lipschitz functions.}
In this section, we introduce new sufficient conditions for the space of Lipschitz functions to have the SD$2$P or the SSD$2$P, which will be convenient to use in the later sections. 

Let $M$ be a pointed metric space. We start by recalling the following lemmas from \cite{HOP2022}.

\begin{lemma}[{\cite[Lemma 3.1]{HOP2022}}]\label{lem: main lemma for detecting the SD2P}
Suppose that, whenever $0<\delta<1$, $n\in\N$, $h_1,\dotsc,h_n\in\Lip_0(M)$ with $\|h_i\|\leq 1-\delta$ for every $i\in\{1,\dotsc,n\}$, 
and $\mu\in \ba$ with only non-negative values, 
there exist a subset $A$ of $M$, elements $u,v\in A$ with $u\neq v$,
and functions $f_1,\dotsc,f_n\in B_{\Lip_0(M)}$ satisfying
\begin{itemize}
\item
$\mu(\Gamma_{A})<\delta$;
\item
$f_i|_{M\setminus A}=h_i|_{M\setminus A}$ for every $i\in\{1,\dotsc,n\}$;
\item
$f_i(u)-f_i(v)\geq (1-\delta)d(u,v)$ for every $i\in\{1,\dotsc,n\}$.
\end{itemize}
Then the space $\Lip_0(M)$ has the SD$2$P.
\end{lemma}

\begin{lemma}[{\cite[Lemma 2.1]{HOP2022}}]\label{lem: main lemma for detecting the SSD2P}
Suppose that, whenever $0<\delta<1$, $n\in\N$,
$h_1,\dotsc,h_n\in\Lip_0(M)$ with $\|h_i\|\leq 1-\delta$ for every $i\in\{1,\dotsc,n\}$, and
$\mu\in \ba$ with only non-negative values,
there exist a subset $A$ of $M$
and functions $f_1,\dotsc,f_n,g\in\Lip_0(M)$ satisfying
\begin{itemize}
\item
$\mu(\Gamma_{A})<\delta$;
\item
$f_i|_{M\setminus A}=h_i|_{M\setminus A}$ for every $i\in\{1,\dotsc,n\}$;
\item
$g|_{M\setminus A}=0$ and $\|g\|\geq 1-\delta$;
\item
$\|f_i\pm g\|\leq1$ for every $i\in\{1,\dotsc,n\}$.
\end{itemize}
Then the space $\Lip_0(M)$ has the SSD$2$P.
\end{lemma}
\noindent
We use these lemmas to define new sufficient conditions for $\Lip_0(M)$ to have the SD$2$P and the SSD$2$P.

First, we give a new sufficient condition for the space $\Lip_0(M)$ to have the SD$2$P. We use this condition in Example \ref{Example2.7HOP} to show that the space $\Lip_0(M)$ has the SD$2$P, where $M$ is the metric space from \cite[Example 2.7]{HOP2022} lacking the seq-LTP.

\begin{defn}[cf. {\cite[p.~572]{MR3803112}} and {\cite[Defintion 3.2]{HOP2022}}]\label{def:FLTP}
    We say that $M$ has the \emph{function LTP} (\emph{FLTP}) if, given $\mu\in\ba$ with only non-negative values, $\varepsilon>0$, $n\in \mathbb{N}$, and $f_1,\ldots,f_n\in B_{\Lip_0(M)}$, there exist a subset $A$ of $M$ with $\mu(\Gamma_A)<\varepsilon$ and elements $u, v\in A$ with $u\neq v$ such that
    \begin{equation}\label{ineq: FLTP}
        (1-\varepsilon)\bigl(f_i(x)-f_i(y)+d(u,v)\bigr)\leq d(x,u)+d(y,v)
    \end{equation}
    for all $x,y\in M\setminus A$ and $i\in\{1,\ldots,n\}$.
\end{defn}
\noindent Clearly the FLTP follows from the seq-LTP. On the other hand, we show below that the metric spaces $M$ in Examples~\ref{Example2.7HOP} and \ref{ex: ell_1 subspace without seq-LTP} both have the FLTP but lack the seq-LTP.
It is possible to directly verify that the FLTP implies the LTP. However, one also obtains this fact from the following theorem together with \cite[Theorem 3.1]{MR3803112}.

\begin{thm}[cf. {\cite[Theorem 3.1]{MR3803112}} and {\cite[Theorem 3.3]{HOP2022}}]\label{thm:FLTP_SD2P}
     If $M$ has the FLTP, then $\Lip_0(M)$ has the SD2P.
\end{thm}
\begin{proof}
 We use Lemma~\ref{lem: main lemma for detecting the SD2P}. Let $0<\delta<1$, $n\in \mathbb{N}$, $h_1,\dotsc, h_n\in \Lip_0(M)$ with $\|h_i\|\leq 1-\delta$ for every $i\in \{1,\ldots, n\}$, and let $\mu\in \ba$ with only non-negative values. By Lemma~\ref{lem: main lemma for detecting the SD2P},
it suffices to find a subset $A$ of $M$ and functions $f_1,\dotsc,f_n\in\Lip_0(M)$
satisfying the conditions in that lemma.

Let subset $A$ of $M$ and points $u, v\in A$ satisfy the conditions in Definition \ref{def:FLTP} with $\varepsilon=\delta$. We may and do assume that $0\notin A$. Fix $i \in \{1,\dotsc, n\}$ and define $f_i\colon M\to \mathbb{R}$ by $f_i|_{M\setminus A} = h_i|_{M\setminus A}$,
\begin{align*}
    f_i(u) = \inf_{x\in M\setminus A}\bigl(f_i(x)+d(x,u)\bigr),
\end{align*}
and for all $y\in A\setminus \{u\}$,
\begin{align*}
    f_i(y) = \sup_{x\in (M\setminus A)\cup \{u\}}\bigl(f_i(x)-d(x,y)\bigr).
\end{align*}
Since $\|f_i\|\leq 1$, it remains to show that $f_i(u)-f_i(v)\geq (1-\delta)d(u,v)$. If $f_i(v) = f_i(u)-d(u,v)$, then the inequality holds. Otherwise,
\begin{align*}
f_i(u)-f_i(v) &= \inf_{x,y\in M\setminus A}\big(f_i(x)+d(x,u)-f_i(y)+d(y,v)\big)\geq (1-\delta) d(u,v).
\end{align*}
\end{proof}

Next, we give a sufficient condition for the space $\Lip_0(M)$ to have the SSD$2$P. We use this condition to prove Proposition \ref{prop:l1_subset_SSD2P}, which shows that for all infinite metric subspaces $M$ of $\ell_1$, which do not necessarily have the seq-(S)LTP (see Example \ref{ex: ell_1 subspace without seq-LTP}), the space $\Lip_0(M)$ has the SSD$2$P.

\begin{defn}[cf. {\cite[Definition 1.3]{MR4026495}} and {\cite[Definition 2.2]{HOP2022}}]\label{def:FSLTP}
    We say that $M$ has the \emph{function SLTP} (\emph{FSLTP}) if, given $\mu\in\ba$ with only non-negative values, $\varepsilon>0$, $n\in \mathbb{N}$, and $f_1,\ldots,f_n\in B_{\Lip_0(M)}$, there exist a subset $A$ of $M$ with $\mu(\Gamma_A)<\varepsilon$ and elements $u, v\in A$ with $u\neq v$ such that the inequalities (\ref{ineq: FLTP}) and
 \begin{equation}\label{ineq: FSLTP}
     \begin{aligned}
    &(1-\varepsilon)\bigl(f_i(x)-f_i(y)+f_j(z)-f_j(w)+2d(u,v)\bigr)\\
    &\qquad\qquad \qquad \leq d(x,u)+d(y,u)+d(z,v)+d(w,v)
    \end{aligned}
\end{equation}
hold for all $x,y,z,w \in M\setminus A$ and $i,j\in\{1,\ldots,n\}$.
\end{defn}

\noindent Clearly the FLTP follows from the FSLTP which in turn follows from the seq-SLTP. On the other hand, the metric space $M$ in Example~\ref{Example2.7HOP} has the FLTP but lacks the FSLTP, while the metric space $M$ in Example~\ref{ex: ell_1 subspace without seq-LTP} has the FSLTP but lacks the seq-SLTP and even the seq-LTP. It is possible to directly verify that the FSLTP implies the SLTP. However, one can also obtain this fact from the following theorem together with \cite[Theorem 2.1]{MR4026495}.

\begin{thm}[cf. {\cite[Theorem~2.1]{MR4026495}} and {\cite[Theorem~2.3]{HOP2022}}]\label{thm:FSLTP_SSD2P}
     If $M$ has the FSLTP, then $\Lip_0(M)$ has the SSD2P.
\end{thm}
\begin{proof}
    We use Lemma~\ref{lem: main lemma for detecting the SSD2P}. Let $0<\delta<1$, $n\in \mathbb{N}$, $h_1,\dotsc, h_n\in \Lip_0(M)$ with $\|h_i\|\leq 1-\delta$ for every $i\in \{1,\ldots, n\}$, and let $\mu\in \ba$ with only non-negative values. By Lemma~\ref{lem: main lemma for detecting the SSD2P},
it suffices to find a subset $A$ of $M$ and functions $f_1,\dotsc,f_n,g\in\Lip_0(M)$
satisfying the conditions in that lemma.

Let subset $A$ of $M$ and points $u, v\in A$ satisfy the conditions in Definition \ref{def:FSLTP} with $\varepsilon=\delta$. We may and do assume that $0\notin A$. Setting
\[
r_0=\frac{1}{2}\inf\bigl\{d(x,u)+d(y,u)-\big(h_i(x)-h_i(y)\big)\colon x,y\in M\setminus A,\ i\in \{1,\dotsc,n\}\bigr\}
\]
and
\[
s_0=\frac{1}{2}\inf\bigl\{d(z,v)+d(w,v)-\big(h_i(z)-h_i(w)\big)\colon z,w\in M\setminus A,\ i\in \{1,\dotsc,n\}\bigr\}
\]
one has $r_0+s_0\geq(1-\delta)d(u,v)$. Thus, there exist $r,s\geq 0$ with $r \leq  r_0$ and $s\leq s_0$ such that
\[
r+s=(1-\delta) d(u,v).
\]
We may assume that $r>0$. By defining $f_1,\dotsc, f_n, g\in B_{\Lip_0(M)}$ as in the proof of {\cite[Theorem~2.3]{HOP2022}} (cf. \cite[Proof of Theorem~2.1]{MR4026495}) and following the rest of that proof nearly word-for-word, it follows that $\|f_i\pm g\|\leq 1$ for all $i\in \{1,\dotsc,n\}$. Let us pinpoint the differences with that proof. From our definitions of $r$ and $s$, it follows immediately that for any $x,y\in M\setminus A$ and $i\in \{1,\dotsc,n\}$,
\[
 h_i(x)+d(x,u)-\bigl( h_i(y)-d(y,u)\bigr)\geq 2r,
\]
and for any $z,w\in M\setminus A$ and $i\in \{1,\dotsc,n\}$,
\[
 h_i(z)+d(z,v)-\bigl( h_i(w)-d(w,v)\bigr)\geq 2s.
\]
Additionally, for any $x,y\in M\setminus A$ and $i\in \{1,\dotsc,n\}$, we have
\begin{align*}
 h_i(x)+d(x,u)-\bigl( h_i(y)-d(y,v)\bigr)\geq (1-\delta)d(u,v)= r+s.
\end{align*}
\end{proof}

\section{Separating diameter two properties from their weak-star counterparts in spaces of Lipschitz functions}\label{sec: w*-d2p and d2p separation}
In this section, we show that there exist spaces of Lipschitz functions, and thus dual Banach spaces, with the $w^*$-SSD$2$P but without the SSD$2$P. To our knowledge, these are first known examples of such spaces, thus answering {\cite[Question 6.2]{MR3917942}}.

Let $\mathcal{U}$ be a free ultrafilter on $\mathbb{N}$. We make use of the following observation: given a Banach space $X$ and a sequence $x_n\in B_X$, it is straightforward to verify that the continuous linear function defined by
\[
 X^* \ni f\longmapsto\lim_{\mathcal{U}} f(x_n)\in \mathbb R,
\]
has a norm at most $1$.

We start with the metric space $M$ introduced in {\cite[Example 2.7]{HOP2022}}, where it was shown that this space has the SLTP but not the seq-LTP. Consequently, $\Lip_0(M)$ has the $w^*$-SSD$2$P. However, it remained open whether the space $\Lip_0(M)$ has the SD$2$P or even the SSD$2$P. We prove that $\Lip_0(M)$ has the SD$2$P but not the SSD$2$P.

\begin{ex}[{\cite[Example 2.7]{HOP2022}}]\label{Example2.7HOP}
Let $M = \{a_k, b_k, c_k\colon k\in \N\}$ be a metric space where for every $k\in \N$,
\[d(a_k, c_k)= 2,\]
for all $k,l\in \N$ with $k<l$, 
\[d(a_k, b_l) = d(b_k, b_l) = d(c_k, b_l)=2,\] and the distance between two different elements is equal to $1$ in all other cases (see Figure~\ref{Figure_1}).

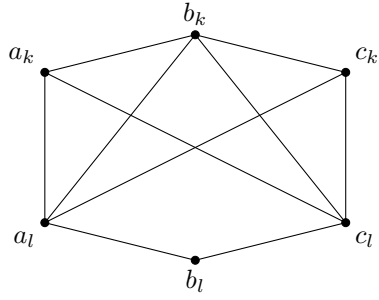
\begin{figure}[ht]
\begin{tikzpicture}
\filldraw[black] (-2, 0) circle (1.5pt) node[below left] {$a_l$};
\filldraw[black] (0, -0.5) circle (1.5pt) node[below] {$b_l$};
\filldraw[black] (2, 0) circle (1.5pt) node[below right] {$c_l$};
\filldraw[black] (-2, 2) circle (1.5pt) node[above left] {$a_k$};
\filldraw[black] (0, 2.5) circle (1.5pt)node[above] {$b_k$};
\filldraw[black] (2, 2) circle (1.5pt) node[above right] {$c_k$};
\draw (-2, 0) -- (0,-0.5);
\draw (-2, 0) -- (-2,2);
\draw (-2, 0) -- (0,2.5);
\draw (-2, 0) -- (2,2);
\draw (2, 0) -- (0,-0.5);
\draw (2, 0) -- (-2,2);
\draw (2, 0) -- (0,2.5);
\draw (2, 0) -- (2,2);
\draw (-2, 2) -- (0,2.5);
\draw (2, 2) -- (0,2.5);
\end{tikzpicture}

\caption{A representation of the metric space M in Example~\ref{Example2.7HOP}. The
distances between points connected by a straight line segment are 1, the
distances between other different points are 2.
} \label{Figure_1}
\end{figure}

First, we show that the space $\Lip_0(M)$ does not have the SSD2P.
Let $0 = b_1$. Define
\[F= m_{a_1,c_1}\in S_{\Lip_0(M)^*}\]
and $G\in B_{\Lip_0(M)^*}$ by 
\[G \colon f \longmapsto \lim_{\mathcal{U}}f(m_{a_{2n-1},a_{2n}}),\]
where $\mathcal{U}$ is a free ultrafilter on $\mathbb{N}$.
Note that $G\in S_{\Lip_0(M)^*}$ because for $f\in S_{\Lip_0(M)}$, defined by
\[
    f(x)=
    \begin{cases}
    1,&x=a_{2n-1},\; n\in\mathbb{N};\\
    0,& \text{otherwise},
    \end{cases}
\]
we have $G(f)=1$.

Let $0<\alpha<\frac{1}{11}$ and let $S_1=S(F,\alpha)$, $S_2=S(G,\alpha)$ be slices of $B_{\Lip_0(M)}$. Assume by contradiction that $\Lip_0(M)$ has the SSD2P. Let $f_1\in S_1$, $f_2\in S_2$, and $g\in B_{\Lip_0(M)}$ with $\|g\|\geq 1-\alpha$ be such that $f_i\pm g \in S_i$ for $i\in \{1,2\}$. Then for any $i\in \{1,2\}$,
\begin{equation}\label{ineq: f_i pm g leq dxy}
    |f_i(x)-f_i(y)|+|g(x)-g(y)|\leq d(x,y)\qquad \text{for all $x,y\in M$.}
\end{equation}
Observe from $f_1\in S_1$ that 
\[
    f_1(a_1) > 1- 2\alpha\quad \text{and}\quad  f_1(c_1) < -1+ 2\alpha.
\] 
It follows from \eqref{ineq: f_i pm g leq dxy} with $i=1$ that for all $n \geq 2$,
\[
    |f_1(a_n)| < 2\alpha\quad \text{and}\quad  |f_1(c_n)| < 2\alpha,
\] 
and that for all $n\in \mathbb{N}$,
\[
     |g(a_n)| < 2\alpha\quad \text{and}\quad  |g(c_n)| < 2\alpha.
\]
Consequently, since $\|g\|\geq 1-\alpha$, there exists $k \in \mathbb{N}$ such that $|g(b_k)| \geq 1 - 3\alpha$. From \eqref{ineq: f_i pm g leq dxy} with $i=2$ we have that for all $l\geq k$,
\begin{align*}
    |f_2(a_l)-f_2(a_{l+1})|&\leq |f_2(a_l)-f_2(b_{k})|+|f_2(a_{l+1})-f_2(b_{k})| \\
    &\leq 2 -|g(a_l)-g(b_{k})|-|g(a_{l+1})-g(b_{k})| < 10\alpha. 
\end{align*} 
Therefore, $f_2\notin S_2$ because
\[
G(f_2)\leq \limsup_n f_2(m_{a_{2n-1},a_{2n}})\leq 10\alpha<1-\alpha.
\]

Next, we show that the space $\Lip_0(M)$ has the SD2P.
   By Theorem \ref{thm:FLTP_SD2P} it suffices to show that $M$ has the FLTP. Fix $\mu\in\ba$ with only non-negative values, $\varepsilon>0$, $n\in\mathbb N$, and $f_1,\ldots,f_n\in B_{\Lip_0(M)}$. Note that if $|f_i(a_k)-f_i(c_k)|\geq 1$ for some $i\in \{1,\dotsc, n\}$ and $k\in \mathbb{N}$, then $|f_i(a_l)-f_i(c_l)|\leq 1$ for all $l\in \mathbb{N}\setminus\{k\}$. Thus, we can find $k\in \mathbb{N}$ so that $\mu(\Gamma_{\{a_k,b_k\}})<\varepsilon$ and
   \[|f_i(a_l)-f_i(c_l)|\leq 1\]
   
   for all $i\in\{1,\ldots,n\}$ and $l\geq k$. Let $A=\{a_k,b_k\}$, $u=a_k$, $v=b_k$, $x,y\in M\setminus A$, and $i\in \{1,\dotsc,n\}$. 
   
   It remains to show that inequality \eqref{ineq: FLTP} holds. If $f_i(x)-f_i(y)\leq 1$, then we are done. Assume now that $f_i(x)-f_i(y)>1$. It suffices to show that $d(x,u)=2$ or $d(y,v)=2$. If $d(y,v)=1$, then $y\in \{a_{l+1},c_l\colon l\geq k\}$. Since $f_i(x)-f_i(y)>1$, we have $d(x,y)=2$ and $x\notin \{a_l,c_l\colon l\geq k\}$ by our choice of $k$. Thus, $x=b_l$ for some $l>k$, and therefore $d(x,u)=2$.
\end{ex}

It is natural to wonder if there exists a metric space $M$ for which $\Lip_0(M)$ has the $w^*$-SD$2$P but lacks the SD$2$P. By \cite[Theorem 2.1 and Corollary 2.2]{MR3150166}, this question is equivalent to \cite[Problem 1.1]{MR4093788}(see also \cite[p.~1681]{MR3985517}). We solve this question by giving an example of a metric space $M$ for which $\Lip_0(M)$ has the $w^*$-SSD$2$P but lacks the slice-D$2$P.

\begin{ex}\label{Example_uus}
    Let $M= \{a_1, a_2, b_k, c_k\colon k\in \mathbb{N}\}$ be a metric space where for all $k,l\in \mathbb{N}$ with $k\leq l$,
    \begin{align*}
        d(a_1, b_{2k-1})=d(a_2,b_{2k}) = d(c_k, b_l) = 1,
    \end{align*}
    and the distance between two different elements is equal to $2$ in all other cases (see Figure~\ref{Figure_2}).
    \begin{figure}[ht]
\begin{tikzpicture}
\filldraw[black] (1.5, 3.5) circle (1.5pt) node[above] {$a_1$};
\filldraw[black] (4.5, 3.5) circle (1.5pt) node[above] {$a_2$};
\filldraw[black] (0, 2.5) circle (1.5pt) node[above left = -0.5mm] {$b_{2k-1}$};
\filldraw[black] (2, 2.5) circle (1.5pt)node[above left  = -0.5mm] {$b_{2k}$};
\filldraw[black] (4, 2.5) circle (1.5pt) node[above right  = -0.5mm] {$b_{2l-1}$};
\filldraw[black] (6, 2.5) circle (1.5pt)node[above right = -0.5mm] {$b_{2l}$};

\draw (1.5, 3.5) -- (0, 2.5);
\draw (1.5, 3.5) -- (4, 2.5);
\draw (4.5, 3.5) -- (2, 2.5);
\draw (4.5, 3.5) -- (6, 2.5);

\filldraw[black] (0, 1) circle (1.5pt) node[below] {$c_{2k-1}$};
\filldraw[black] (2, 1) circle (1.5pt) node[below] {$c_{2k}$};
\filldraw[black] (4, 1) circle (1.5pt) node[below] {$c_{2l-1}$};
\filldraw[black] (6, 1) circle (1.5pt)node[below] {$c_{2l}$};
\draw (0, 1) -- (0, 2.5);
\draw (0, 1) -- (2, 2.5);
\draw (0, 1) -- (4, 2.5);
\draw (0, 1) -- (6, 2.5);
\draw (2, 1) -- (2, 2.5);
\draw (2, 1) -- (4, 2.5);
\draw (2, 1) -- (6, 2.5);
\draw (4, 1) -- (4, 2.5);
\draw (4, 1) -- (6, 2.5);
\draw (6, 1) -- (6, 2.5);
\end{tikzpicture}
\caption{A representation of the metric space M in Example~\ref{Example_uus}. The
distances between points connected by a straight line segment are 1, the
distances between other different points are 2.
}\label{Figure_2}
\end{figure}
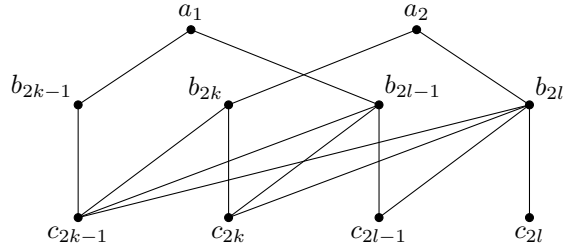

First, we verify that the space $\Lip_0(M)$ has the $w^*$-SSD$2$P. It suffices to show that $M$ has SLTP. Let $N$ be a finite subset of $M$. There exists $k\in\mathbb{N}$ such that $d(x,c_k)=2$ and $d(x,c_{k+1})=2$ for every $x\in N$. Take $u=c_k$ and $v=c_{k+1}$, and observe that inequalities \eqref{ineq: LTP} and \eqref{ineq: SLTP} hold for all $x,y,z,w\in N$. 

Next, we show that the space $\Lip_0(M)$ does not have the slice-D$2$P. Let $0=c_1$. Define $F\in B_{\Lip_0(M)^*}$ by 
\[
    F \colon f \longmapsto \lim_{\mathcal{U}}f(m_{b_{2n-1},b_{2n}}),
\]
where $\mathcal{U}$ is a free ultrafilter on $\mathbb{N}$, and let 
\[
    G=\frac12(m_{a_1,a_2}+F).
\]
Note that $G\in S_{\Lip_0(M)^*}$ because for $f\in S_{\Lip_0(M)}$, defined by
\[
    f(x)=\begin{cases}1,&x=a_1 \text{ or } x=b_{2n-1},\; n\in\mathbb{N};\\
    -1,& x=a_2 \text{ or } x=b_{2n},\; n\in\mathbb{N};\\
    0,&x=c_n,\; n\in\mathbb{N},
    \end{cases}
\]
we have $G(f)=1$. 

Let $0<\varepsilon\leq \alpha<\frac{1}{13}$ and let $S=S(G,\alpha)$ be a slice of $B_{\Lip_0(M)}$. Let $f\in S$. Since $F(f)>1-2\alpha$, we have $\limsup_n f(m_{b_{2n-1},b_{2n}})> 1-2\alpha$. This implies
\[
    \limsup_n f(b_{n})> 1-4\alpha \quad \text{and}\quad \liminf_n f(b_{n})< -1+4\alpha,
 \]
 and thus $|f(c_n)|< 4\alpha$ for all $n\in\mathbb{N}$. Since $m_{a_1,a_2}(f)>1-2\alpha$, we have  
    \[
    f(a_1)> 1-4\alpha \quad \text{and}\quad f(a_2) < -1+4\alpha.
    \]
    Hence, for every $n\in\mathbb{N}$,
    \[
    f(b_{2n-1})> -4\alpha \quad \text{and}\quad f(b_{2n})< 4\alpha.
    \]
    As these bounds hold for all $f\in S$, we have 
\[
    \diam(S)\leq  1+12\alpha<2-\varepsilon.
\]
    
\end{ex}

\section{Strong diameter two properties in spaces of Lipschitz functions over infinite metric subspaces of \texorpdfstring{$\ell_1$}{l1}}

In this section, we show that the space $\Lip_0(M)$ has the SSD$2$P for any infinite metric subspace $M$ of $\ell_1$. We first provide an example of such a metric space $M$ lacking the seq-LTP. Therefore, the seq-LTP (or the seq-SLTP) for $M$ is not necessary for $\Lip_0(M)$ to have the SD$2$P or the SSD$2$P. The question of whether all infinite metric subspaces of $\ell_1$ have the seq-(S)LTP was initially raised by Antonín Procházka in 2022 during the PhD thesis defense of the third author.

\begin{ex}\label{ex: ell_1 subspace without seq-LTP}
    The metric subspace $M=\{e_i+e_j\colon i,j\in \mathbb{N}\}$ of $\ell_1$ does not have the seq-LTP. 
    
    Let $A$ be a subset of $M$. Let $u = e_i+e_j,$ $v= e_k+e_l\in A$ with $u\neq v$ for some $i,j,k,l\in \mathbb{N}$. Assume that inequality \eqref{ineq: LTP} holds for all $x,y\in M\setminus A$. We consider two distinct cases to show that there is an $m\in \mathbb{N}$ for which there are at most two different $n\in\mathbb{N}$ such that $e_m+e_n\not\in A$. 
    
    If $d(u,v)=2$, then $i\in \{k,l\}$ or $j\in \{k,l\}$. Without loss of generality, assume that $i=k$. If $e_j+e_m\not\in A$ for some $m\in \mathbb{N}\setminus\{l\}$, then 
    \[
        d(u,e_j+e_m)=2,\quad d(v,e_l+e_n)= 2,\quad \text{and} \quad d(e_j+e_m,e_l+e_n)=4
    \] 
    for every $n\in \mathbb{N}\setminus\{m,j\}$.
    Therefore, $e_l+e_n\in A$ for all $n\in \mathbb{N}\setminus\{m,j\}$.
    
    If $d(u,v)=4$ and $e_j+e_m\not\in A$ for some $m\in \mathbb{N}$, then 
    \[
    d(u,e_j+e_m)=2\quad \text{and}\quad d(v,e_l+e_n)\leq 2
    \]
    for every $n\in \mathbb{N}\setminus\{m,j\}$. Thus, $e_l+e_n\in A$ for all $n\in \mathbb{N}\setminus\{m,j\}$.

    Now consider subsets $A_m$ of $M$ and $u_m,v_m\in A_m$ with $u_m\neq v_m$ and $m\in \{1,\dotsc,6\}$ that satisfy inequality \eqref{ineq: seq-LTP} for all $x,y\in M\setminus A_m$ and $m\in \{1,\dotsc,6\}$. By previous analysis, $A_1,\dotsc, A_6$ can not be pairwise disjoint. Therefore, $M$ does not have the seq-LTP.
\end{ex}

    To show that $\Lip_0(M)$ has the SSD$2$P for every infinite metric subspace $M$ of $\ell_1$, it is convenient to use the following result.

\begin{lemma}\label{lem: exists A with mu(A) < delta}
    Let $E$ be any set, $A_1,A_2,\dotsc$ subsets of $E$, and $\mu\in ba(E)$ with only non-negative values. If there exists $n\in \mathbb{N}$ such that for any pairwise different $m_1,\dotsc, m_n \in \mathbb{N}$, the intersection $A_{m_1}\cap\dotsb\cap A_{m_n}$ is empty, then for any $\delta>0$, there exists at most finitely many such $m\in \mathbb{N}$ that $\mu(A_m)\geq \delta$.
\end{lemma}

\begin{proof}
Assume by contradiction that for some $\delta>0$ there exists a subsequence $(B_m)$ of $(A_m)$ such that $\mu(B_m)\geq \delta$ for all $m\in \mathbb{N}$. 

We first show that there exist $\delta_1> 0$ and $i_m, j_m\in \mathbb{N}$ so that  $i_m<j_m<i_{m+1}$ and $\mu(B_{i_m}\cap B_{j_m})> \delta_1$ for all $m\in \mathbb{N}$. We may and do assume that $\mu(E)=1$. Choose $k\in \mathbb{N}$ so that $k\delta>1+\delta$, fix $m\in \mathbb{N}$, and let 
\[
\beta =\max\big\{\mu(B_{i}\cap B_{j})\colon i,j\in \{km+1,\dotsc, km+k\},\ i\neq j\big\}.
\] 
Observe that $\beta > \delta/(2^k-k)$ because
    \begin{align*}
        1=\mu(E)&\geq\mu(B_{km+1}\cup \dotsb \cup B_{km+k}) \\ 
        &\geq \sum_{i=1}^k \mu(B_{km+i})-(2^k-k)\beta \\
        &\geq k\delta-(2^k-k)\beta\\
        &>1+\delta -(2^k-k)\beta.
    \end{align*}
Taking $i_m,j_m\in \{km+1,\dotsc, km+k\}$ with $i_m< j_m$ so that $\mu(B_{i_m}\cap B_{j_m}) =\beta$, we have that $\mu(B_{i_m}\cap B_{j_m})> \delta_1$, where $\delta_1=\delta/(2^k-k)$.

Now let $C_m=B_{i_m}\cap B_{j_m}$ and note that $\mu(C_m)\geq \delta_1$ for all $m\in \mathbb{N}$. Replacing $B_m$ with $C_m$ and $\delta$ with $\delta_1$ in the previous argument, we get that there exist $\delta_2>0$ and $k_m,l_m\in \mathbb{N}$ so that $k_m<l_m<k_{m+1}$ and $\mu(C_{k_m}\cap C_{l_m})> \delta_2$ for all $m\in \mathbb{N}$. Repeating the same argument as many times as is needed, we get that there exist pairwise different $m_1,\dotsc,m_n\in \mathbb{N}$ such that $\mu(A_{m_1}\cap\dotsb \cap A_{m_n})>0$. However, by our assumption, $A_{m_1}\cap \dotsb \cap A_{m_n}=\emptyset$.
\end{proof}

\begin{prop}\label{prop:l1_subset_SSD2P}
    If $M$ is an infinite metric subspace of $\ell_1$, then $\Lip_0(M)$ has the SSD$2$P.
\end{prop}

\begin{proof}
    If $M$ is unbounded or not uniformly discrete, then $\Lip_0(M)$ has the SSD$2$P by \cite[Theorem 2.2]{MR4093788} ($M$ even has the seq-SLTP by \cite[Proposition 2.4]{HOP2022}). Assume that $M$ is bounded and uniformly discrete, i.e. there exist $r,R> 0$ such that for all $x,y\in M$ with $x\neq y$, we have $r\leq d(x,y)\leq R$.
    Without loss of generality, we assume that $M$ contains the element $(0,0,\dotsc)$, which we assign as the $0$ of the pointed metric space $M$.
    
     By Theorem~\ref{thm:FSLTP_SSD2P}, it suffices to show that $M$ has the FSLTP. Let $\mu\in \ba$ with only non-negative values, $\varepsilon\in (0,1)$, $n\in \mathbb{N}$, $f_1,\dotsc, f_n\in B_{\Lip_0(M)}$, and let $K_0=1$. We inductively define $K_m\in \mathbb N$, $u_m,v_m\in M$, and a subset $A_m$ of $M$ for all $m\in \mathbb{N}$. Assume that we have $K_{m-1}\in \mathbb{N}$ for some $m\in\mathbb{N}$. Choose $K_{m}\in \mathbb{N}$ with $K_m>K_{m-1}$ and $u_m,v_m\in M\setminus\{0\}$ with $u_m\neq v_m$ so that 
    \[
    \sum_{k\in I_m} |u_m(k)|\geq \|u_m\|-\frac{\varepsilon r}{2}\quad \text{and} \quad \sum_{k\in I_m} |v_m(k)| \geq \|v_m\|-\frac{\varepsilon r}{2},
    \]
    where $I_m=\{K_{m-1},\dotsc, K_m-1\}$. Define
    \[
    A_m = \Big\{x\in M\colon \sum_{k\in I_m} |x(k)| \geq \frac{\varepsilon r}{2}\Big\}.
    \]

    Note that for every $x\in M$, 
    \[
    \big|\{m\in \mathbb N\colon x\in A_m\}\big|\leq \frac{2R}{\varepsilon r}.
    \]
     By Lemma~\ref{lem: exists A with mu(A) < delta}, there exist at most finitely many such $m\in \mathbb{N}$ that $\mu(\Gamma_{1,A_m})\geq \varepsilon/2$ or $\mu(\Gamma_{2,A_m})\geq \varepsilon/2$, where 
   \[
    \Gamma_{1,A_m}=\big\{(x,y)\in\widetilde{M}\colon x\in A_m\big\}
    \quad\text{and}\quad
    \Gamma_{2,A_m}=\big\{(x,y)\in\widetilde{M}\colon y\in A_m\big\}.
    \]
    Therefore, there exists $m\in \mathbb{N}$ such that $\mu(\Gamma_{A_m})<\varepsilon$. Fix such $m\in \mathbb{N}$ and let $A=A_m$, $u = u_m, v=v_m$, and $I=\{K_{m-1},\dotsc, K_m-1\}$. 

    Note that if inequalities \eqref{ineq: LTP} and \eqref{ineq: SLTP} hold for all $x,y,z,w\in M\setminus A$, then inequalities \eqref{ineq: FLTP} and \eqref{ineq: FSLTP} hold for all $x,y,z,w\in M\setminus A$ and $i,j\in \{1,\dotsc,n\}$ because for all $i\in \{1,\dotsc,n\}$ and $x,y\in M$,
     \[
        f_i(x)-f_i(y)\leq d(x,y).
     \]
    Let $x\in M\setminus A$. From 
    \begin{align*}
        \sum_{k\in I} |x(k)| < \frac{\varepsilon r}{2} \quad \text{and} \quad \sum_{k\in I} |u(k)| \geq \|u\|-\frac{\varepsilon r}{2}
    \end{align*}
    we have that
    \[
        \|x-u\|\geq\|x\| +\|u\|-2\varepsilon r\geq (1-\varepsilon)\big(\|x\|+\|u\|\big).
    \]
    Analogously, we get that $\|x-v\|\geq (1-\varepsilon)\big(\|x\|+\|v\|\big)$. It is now easy to check that inequalities \eqref{ineq: LTP} and \eqref{ineq: SLTP} hold for all $x,y,z,w\in M\setminus A$. Therefore, $M$ has the FSLTP.
\end{proof}

\section{Open questions}
We conclude the paper by stating some open questions. 

\begin{itemize}
    \item 
In section \ref{sec: w*-d2p and d2p separation}, we gave first known examples of dual Banach spaces with the $w^*$-SSD$2$P but without the SSD$2$P. It remains an open question whether such bidual spaces exist. This enquiry was kindly presented to us by Vegard Lima and Trond A. Abrahamsen.
\begin{qu}
    Is there a Banach space $X$ with the SSD$2$P such that $X^{**}$ lacks the SSD$2$P?
\end{qu}
\noindent To provide an affirmative answer, it would suffice to show that $\mathcal{F}(M)$ is a dual space for $M$ from Example \ref{Example2.7HOP} or \ref{Example_uus}. In fact, if the previous claim is true for $M$ from Example \ref{Example_uus}, then the predual of $\mathcal{F}(M)$ is an example of a Banach space $X$ with the SSD$2$P so that its bidual $X^{**}=\Lip_0(M)$ lacks the slice-D$2$P.
    \item 
A stronger property than the traditional octahedrality in Banach spaces, called decomposable octahedrality, was introduced by the third author in \cite{MR4233633}.
It was shown that a Banach space $X$ is decomposably octahedral whenever its dual space $X^*$ has the $w^*$-SSD$2$P, and that the converse implication holds if $X$ is a Lipschitz-free space. Whether this converse implication holds for every Banach space $X$ remains an open question.
Our study prompts a more concrete question.

\begin{qu}
Does $\Lip_0(M)$ have the SSD$2$P whenever $\Lip_0(M)^*$ is decomposably octahedral? 
\end{qu}
\noindent The metric space $M$ from Example~\ref{Example2.7HOP} could yield a negative answer.

    \item 
The $w^*$-SD$2$P and $w^*$-SSD$2$P in $\Lip_0(M)$ have convenient characterisations in the form of the LTP and the SLTP in $M$, respectively. It is natural to ask if there are similar metric characterisations for other diameter $2$ properties.
\begin{qu}
    Do all diameter 2 properties have metric descriptions for spaces of Lipschitz functions?
\end{qu}

    \item 
Finally, we recall a question posed in \cite[p.~3]{HOP2022}, which remains unanswered in this subsequent study.
\begin{qu}
Are the ($w^*$-)slice-D$2$P and the ($w^*$-)D$2$P distinct properties in spaces of Lipschitz functions?
\end{qu}
\noindent The answer to this question is known in the broader context of Banach spaces \cite{MR3334951}.
\end{itemize}

\section*{Acknowledgments}
This work was supported by the Estonian Research
Council grant (PRG1901).

The authors thank Trond A.~Abrahamsen, Vegard Lima, and Yo\"el Perreau for helpful comments on the content and presentation of this paper.

\bibliographystyle{amsplain}
\bibliography{references}

\end{document}